\documentclass[a4paper]{amsart}
\usepackage[utf8]{inputenc}
\usepackage[a4paper,margin=2.5cm]{geometry}
\usepackage{amsmath}
\usepackage{amssymb,amsthm,mathtools,abraces,comment}
\usepackage{wasysym}

\usepackage[dvipsnames]{xcolor}
\usepackage[hypertexnames=false,hyperfootnotes=false,colorlinks=true,linkcolor=blue,%
citecolor=purple,filecolor=magenta,urlcolor=cyan,unicode,linktocpage=true]{hyperref}

\usepackage[all]{xy}

\newtheorem{theorem}{Theorem}[section]     
 
\newtheorem{lemma}[theorem]{Lemma} 
\newtheorem{corollary}[theorem]{Corollary} 
 
\newtheorem{conjecture}[theorem]{Conjecture} 
\theoremstyle{definition}

\theoremstyle{remark}
\newtheorem{remark}[theorem]{Remark}

\newcommand{\oM}{\overline{\mathcal{M}}}

\newcommand{\gl}{\mathrm{gl}}
\newcommand{\Coef}{\mathrm{Coef}}

\def\res{\mathrm{res}}
\def\DR{\mathrm{DR}}
\def\DZ{\mathrm{DZ}}

\def\sfB{{\mathsf{B}}}

\begin{document}
\title{A conjectural formula for $\lambda_g \DR_g(a,-a)$ is true in Gorenstein quotient}

\author[D.~Gubarevich]{D. Gubarevich}

\address[D. Gubarevich]{Faculty of Mathematics, National Research University Higher School of Economics,
Usacheva 6, 119048, Moscow, Russia}

\email{danilphys180916@mail.ru}

\begin{abstract}
In this note we prove that a conjectural formula for the class $\lambda_g \DR_g(a,-a)\in R^{2g}(\oM_{g,2})$ proposed recently in a work of ~\cite{BIS21} is true in the Gorenstein quotient of the ring $R^*(\oM_{g,2})$. As a corollary we prove the strong $\DR \backslash \DZ$ equivalence conjecture for one-point correlators.
\end{abstract}

\maketitle

\tableofcontents 

\section{Introduction}

The class $\lambda_g \DR_g(a,-a)\in R^{2g}(\oM_{g,2})$ has a very transparent geometric nature. It is the intersection of two tautological classes. The first one is the top Chern class 
$\lambda_g=c_g(\mathbb{E})\in R^{g}(\oM_{g,2})$ of the Hodge bundle $\mathbb{E}$ of holomorphic differential forms over moduli space $\oM_{g,2}$ of genus $g$ stable curves with $2$ marked points. The second one is the double ramification cycle $\DR_g(a,-a)\in R^{g}(\oM_{g,2})$, a class with nice history and many interesting properties. For good exposition of  the double ramification cycle see for example ~\cite{JPPZ17}. In this note we prove relatively simple expression for $\lambda_g \DR_g(a,-a)$, $a \in \mathbb{Z}$ in Gorenstein quotient of $R^{*}(\oM_{g,2})$ using its properties under intersection with classes from $R^{1}(\oM_{g,2})$. The similar properties
were proved to hold for the class $\sfB^g$ introduced in ~\cite{BIS21}. They allowed us to find the same recursive way to compute both intersection numbers 
\begin{align}
\Coef_{a^{2g}}\int_{\oM_{g,2}}\DR_{g}(a,-a)\lambda_g  \omega, \quad \int_{\oM_{g,2}}\sfB^g \omega
\end{align}
for any tautological class $\omega$.

The interest in this intersection product class  $\lambda_g \DR_g(a,-a)$  comes from the study of integrable systems. 
Namely, there is a conjecture on coincidence of potentials of two integrable hierarchies (strong $\DR \backslash \DZ$ equivalence), stated in ~\cite{BGR19}. And the proof of this conjecture for one-point correlators follows from our main result as a simple corollary.

\section{The main results}
\subsection{Preliminaries}

Let $\oM_{g,n}$ be the Deligne-Mumford compactified moduli space of stable genus $g$ with $n$ markings. Denote by $R^*(\oM_{g,n})$ its tautological ring, that is defined as the minimal subring of the full cohomology ring containing the unit and that is closed under all pullbacks and pushforwards by gluing and forgetful morphisms. The first Chern classes of line bundles $\mathbb{L}_i$ over $\oM_{g,n}$ formed by the
cotangent lines at the i-th marked point is denoted  by $\psi_i = c_1(\mathbb{L}_i)\in R^1(\oM_{g,n})$. Define $\kappa$-classes by $\kappa_i=\pi_*(\psi_{n+1}^{i+1})\in R^i(\oM_{g,n})$, where $\pi: \oM_{g,n+1} \to \oM_{g,n}$ is the map forgetting the last marking. We will use the notation, explained in  ~\cite{BGR19},
\begin{align*}
\vcenter{\xymatrix@C=10pt{
		&*+[o][F-]{g_1}\ar@{-}[l]*{{\ }_1\ }\ar@{-}[rr]^<<<<{\omega_1} &&*+[o][F-]{g_2}\ar@{-}[rr]^<<<<{\omega_2}&& *+[o][F-]{g_3}\ar@{--}[rrr]^<<<<{\omega_3} && & *+[o][F-]{g_k}\ar@{-}[r]*{{\ }_2}^<<<<{\omega_k} &
}}
\end{align*}
for the boundary class of curves whose dual graphs are such trees with $k-1$ edges and each vertex in general carries a decoration  $\omega_i=\psi_1^{d_1}\psi_2^{d_2}\prod_i\kappa_i^{a_i}\in R^{d_1+d_2+\sum_i ia_i}(\oM_{g_i,2})$. Note here that the double ramification cycle $\DR_g(a_1,\cdots,a_n)\in H^{2g}(\oM_{g,n},\mathbb{Q})$ was proven to be an element of the tautological ring and its intersection with $\lambda_g$ is polynomial in variables $a_1, \dots, a_n$.  The reason of this polynomiality follows from the Hain's formula.  Let $\mathcal{M}^{ct}_{g,n}\subset \oM_{g,n}$ be moduli space of curves with compact Jacobian. Then Hain's formula express the restriction of double ramification cycle to this subspace $\DR_g(a_1,\cdots,a_n)|_{\mathcal{M}^{ct}_{g,n}}$  as a degree $g$ homogenious polynomial in $a_i$'s. But since $\lambda_g|_{\oM_{g,2}\backslash \mathcal{M}^{ct}_{g,2}} = 0$ \footnote{See proof of Lemma (\ref{lem})} then clearly the intersection $\DR_g(a_1,\cdots,a_n)\lambda_g\in R^{2g}(\oM_{g,n})$ is a polynomial class of space $\oM_{g,n}$. The class $\sfB^g$ is defined as 
\begin{align}
\label{bumboo}
\sfB^g = \sum_{k=1}^g (-1)^{k-1} \sum_{\substack{d_1,\dots,d_k \\ g_1,\dots,g_k}} 
\vcenter{\xymatrix@C=10pt{
		&*+[o][F-]{g_1}\ar@{-}[l]*{{\ }_1\ }\ar@{-}[rr]^<<<<{\psi^{d_1}} &&*+[o][F-]{g_2}\ar@{-}[rr]^<<<<{\psi^{d_2}}&& *+[o][F-]{g_3}\ar@{--}[rrr]^<<<<{\psi^{d_3}} && & *+[o][F-]{g_k}\ar@{-}[r]*{{\ }_2}^<<<<{\psi^{d_k}} &
}},
\end{align}
where the sum is taken over all $g_1+\cdots+g_k=g$, $g_1,\dots,g_k\geq 1$, and $d_1+\cdots+d_k+k-1 = 2g$, $d_1,\dots,d_k\geq 0$, with the extra condition that for any $1\leq \ell\leq k-1$ we have $d_1+\cdots+d_\ell +\ell-1\leq 2(g_1+\cdots +g_\ell) -1$.

\subsection{Conjectural formula}

In ~\cite{BIS21} the authors formulated the following
\begin{conjecture}
 $\Coef_{a^{2g}} \DR_{{g}}(a,-a)\lambda_g = \sfB^g$. 
\end{conjecture}
To state the equality of two cohomological classes is an involved problem in general. We prove a less strong result. Namely, that integrals of these classes with any tautological classes of comlementary degrees are equal. In this situation we call such classes equal in Gorenstein quotient. Now we state the 

\begin{theorem}
\label{th}
The conjecture 
\begin{align*}
 \Coef_{a^{2g}}\DR_{g}(a,-a)\lambda_g = \sfB^g
\end{align*}
is true in Gorenstein quotient of  $R^*(\oM_{g,2})$.
\end{theorem}

And, since each semisimple cohomological field theory can be reconstructed from its topological part via Givental-Teleman R-matrix action, it can be represented as a sum over decorated by $\psi$- and $\kappa$-classes stable graphs. 
Following the paper \cite{BGR19}, denote by $A^g_{2g-1}=\pi_*\Coef_{a^{2g}}\DR_{g}(a,-a)\lambda_g$ and $ B^g_{2g-1}=\pi_*\sfB^g \in R^{2g-1}(\oM_{g,1})$, where $\pi: \oM_{g,2}\rightarrow \oM_{g,1}$ is a forgetting morphism. We have the following 
\begin{corollary}
For each semisimple cohomological field theory $c_{g,n}$
\begin{align}
\label{cor}
\int_{\oM_{g,1}} A^g_{2g-1}c_{g,1}=\int_{\oM_{g,1}} B^g_{2g-1}c_{g,1}.
\end{align}
\end{corollary}

\begin{remark}
As we mentioned in the introduction, the statement in corollary means precisely that one-point parts of potential of double ramification hierarchy and reduced potential of Dubrovin-Zhang hierarchy are equal.
\end{remark}

A conjectural formula for $\lambda_g\mathrm{DR}_g(a,-a)$ is true in Gorenstein quotient

\begin{proof}[Proof of Theorem ~\ref{th}]

We are interested in integrals of the form 
\begin{align}
\label{int}
\Coef_{a^{2g}}\int_{\oM_{g,2}}\DR_{g}(a,-a)\lambda_g  \omega, \quad \int_{\oM_{g,2}}\sfB^g \omega,
\end{align}
when $\omega \in  R^{g-1}(\oM_{g,2})$. Our strategy will be the following. Firstly, we show that if $\omega$ is proportional to a class whose generic point is a curve with dual graph containing a cycle($\omega$ contains a cycle, for short), then its intersection numbers with $\DR_{g}(a,-a)\lambda_g$ and $\sfB^g$ cycles are zero. Then we note that the only trees occurring in $\omega$, that contribute nontrivially, are just bamboos like in formula (\ref{bumboo}) for $\sfB^g$. Then we compute both integrals by induction on genus and show that they are equal to the same expressions. 

Let us list the properties of both classes to be used. Denote by 
$$
\gl_1\colon\oM_{g_1,2}\times\oM_{g_2,2}\to\oM_{g_1+g_2, 2}
$$
the gluing map that corresponds to gluing a curve from~$\oM_{g_1,2}$ to a curve from~$\oM_{g_2,2}$ along the second marked point on the first curve and the second  marked point on the second curve. Let $a\in\mathbb{Z}$ then introduce the notation
\begin{align*}
&\DR_{g_1}(a,-a)\boxtimes_1\DR_{g_2}(a,-a):=\\
=&\gl_{1*}\left(\DR_{g_1}(a,-a)\times\DR_{g_2}(a,-a)\right)\in R^{g_1+g_2+1}(\oM_{g_1+g_2,2}).
\end{align*}
We have the
\begin{lemma}
\label{lem}
The $\DR_{g}(a,-a)\lambda_g$ satisfies the properties
\begin{align}
 \label{eq:irrdiv}
	 & \DR_{g}(a,-a)\lambda_g \cdot \vcenter{
	 	\xymatrix@M=2pt@C=10pt@R=0pt{
	 		& &  \\ 
	 		&*+[o][F-]{{g\text{-}1}}\ar@{-}[lu]*{{\ }_1\ }\ar@{-}[ld]*{{\ }_2\ }\ar@{-}@(ur,dr)\ar@{-}@(dr,ur) & \\ & &
	 	}
 	}= 0; 
\\ \label{eq:div}
 	 & \DR_{g}(a,-a)\lambda_g \cdot \vcenter{
 	 	\xymatrix@C=10pt@R=0pt{
 	 		& & \\ 
 	 		&*+[o][F-]{{g_1}}\ar@{-}[lu]*{{\ }_1\ }\ar@{-}[ld]*{{\ }_2\ }\ar@{-}[r] &*+[o][F-]{{g_2}}\\ & & 
 	 	}
 	 }  = 0, & & g_1+g_2 = g, \ g_2\geq 1; 
\\ \label{eq:split}
& \DR_{g}(a,-a)\lambda_g \cdot \vcenter{
\xymatrix@C=10pt@R=10pt{
	&*+[o][F-]{{g_1}}\ar@{-}[l]*{{\ }_1\ }\ar@{-}[r]& *+[o][F-]{{g_2}} \ar@{-}[r]*{{\ }_2}& 
	}
} = \DR_{g_1}(a,-a)\lambda_{g_1} \boxtimes_1 \DR_{g_2}(a,-a)\lambda_{g_2},
& & g_1+g_2 = g, \ g_1,g_2\geq 1,
\end{align}

\end{lemma}
\begin{proof}[Proof of Lemma ~\ref{lem}]
 Let $\mathcal{M}^{ct}_{g,2}\subset \oM_{g,2}$ be moduli space of curves with compact Jacobian. This means in particular that the dual graph of the stratum of such curves contains a cycle.
Notice that there is surjective morphism $\mathbb{E}\to \mathbb{C}\to 0$ of bundles over $\oM_{g,2}\backslash \mathcal{M}^{ct}_{g,2}$
given on fibers by computing residue at the nodal point  $\omega \to \res_x \omega$. Then from Whitney sum formula the top Chern class restricts by zero on curves with noncompact Jacobian $\lambda_g|_{\oM_{g,2}\backslash \mathcal{M}^{ct}_{g,2}} = 0$ and the first property follows. 

At the same time by splitting property of $\DR$-cycle, proved in ~\cite{BSSZ15}, we have 
\begin{align*}
\DR_{{g}}(a,-a)\lambda_g \vcenter{
 	 	\xymatrix@C=10pt@R=0pt{
 	 		& & \\ 
 	 		&*+[o][F-]{{g_1}}\ar@{-}[lu]*{{\ }_1\ }\ar@{-}[ld]*{{\ }_2\ }\ar@{-}[r] &*+[o][F-]{{g_2}}\\ & & 
 	 	}
 	 }  = \DR_{g_1}\left(a,-a,0\right)\lambda_{g_1}  \boxtimes_1\DR_{g_2}\left(\emptyset\right) \lambda_{g_2} =0,
\end{align*}
since $\DR_{g_2}\left(\emptyset\right) \lambda_{g_2}=(-1)^{g_2} \lambda^{2}_{g_2}=0$. 

Here the fact that $\DR_{g}\left(\emptyset\right)=(-1)^{g} \lambda_{g}$ follows from the form of an obstruction bundle for degree zero maps to a point. Namely, for the empty ramification data the virtual fundumental class of stable rubber maps can be easily computed: $\DR_{g}\left(\emptyset\right) = c_{top}(\mathbb{E}^{*})\cap[\oM_{g,n}]$. And the fact that $\lambda^{2}_{g}=0$ follows from the Grothendieck–Riemann–Roch theorem. Finally, the third property is again the splitting property of $\DR$-cycle.
\end{proof}

The similar properties of class $ \sfB^g$ were proved to hold in \cite{BIS21}.
\begin{align}
 \label{eq:Thm-irrdiv}
	 & \sfB^g \cdot \vcenter{
	 	\xymatrix@M=2pt@C=10pt@R=0pt{
	 		& &  \\ 
	 		&*+[o][F-]{{g\text{-}1}}\ar@{-}[lu]*{{\ }_1\ }\ar@{-}[ld]*{{\ }_2\ }\ar@{-}@(ur,dr)\ar@{-}@(dr,ur) & \\ & &
	 	}
 	}= 0; 
\\ \label{eq:Thm-20div}
 	 & \sfB^g \cdot \vcenter{
 	 	\xymatrix@C=10pt@R=0pt{
 	 		& & \\ 
 	 		&*+[o][F-]{{g_1}}\ar@{-}[lu]*{{\ }_1\ }\ar@{-}[ld]*{{\ }_2\ }\ar@{-}[r] &*+[o][F-]{{g_2}}\\ & & 
 	 	}
 	 }  = 0, & & g_1+g_2 = g, \ g_2\geq 1; 
\\ \label{eq:Thm-Inter-1-g1-g2-2}
& \sfB^g \cdot \vcenter{
\xymatrix@C=10pt@R=10pt{
	&*+[o][F-]{{g_1}}\ar@{-}[l]*{{\ }_1\ }\ar@{-}[r]& *+[o][F-]{{g_2}} \ar@{-}[r]*{{\ }_2}& 
	}
} = \sfB^{g_1}\diamond \sfB^{g_2},
& & g_1+g_2 = g, \ g_1,g_2\geq 1 
\end{align}

Now suppose that the tautological class $\omega$  contains a cycle. Then from
(\ref{eq:irrdiv}),(\ref{eq:Thm-irrdiv}) it is clear that the intersection of both classes with $\omega$ is zero. Then observe that if $\omega$  is proportional to the class whose generic point is a curve with dual graph containing a vertex with only one half-edge incident to it then this $\omega$ lies in a stratum with dual graph as in property (\ref{eq:Thm-20div}) and hence again $ \sfB^g \omega=0$. At the same time by the splitting property (\ref{eq:div}) of 
$\DR$-cycle $\DR_{{g}}(a,-a)\lambda_g\omega=0$.

Then we show that the assertion of the theorem ~\ref{th} is true for decorated bumboo

\begin{align*}
\omega = \vcenter{\xymatrix@C=10pt{
		&*+[o][F-]{g_1}\ar@{-}[l]*{{\ }_1\ }\ar@{-}[rr]^<<<<{\omega_1} &&*+[o][F-]{g_2}\ar@{-}[rr]^<<<<{\omega_2}&& *+[o][F-]{g_3}\ar@{--}[rrr]^<<<<{\omega_3} && & *+[o][F-]{g_k}\ar@{-}[r]*{{\ }_2}^<<<<{\omega_k} &
}}
\quad 
g_1+\dots+g_k = g, 
\end{align*}
where each vertex carries a class $\omega_i = \psi_1^{d_1}\psi_2^{d_2}\prod_i\kappa_i^{a_i}$ with $d_1+d_2+\sum_i ia_i = g-1$.
Then we proceed by induction on genus. Suppose we know that the theorem~\ref{th} is true for genera up to $g-1$.\\
We compute

\begin{align*}
&\Coef_{a^{2g}}\int_{\oM_{g,2}}\DR_{g}(a,-a)\lambda_g  \omega\\
&=  \Coef_{a^{2g}} \int_{\oM_{g,2}}  \DR_{{g_1}}(a,-a)\lambda_{g_1}\psi_1^{d_1} \psi_{e^{\prime}_1}^{a_1}\prod_i\kappa_i^{c_i}  \boxtimes_1 
\DR_{{g_2}}(a,-a)\lambda_{g_2} \psi_{e^{\prime\prime}_1}^{b_1}\psi_{e^{\prime}_2}^{a_2}\prod_i\kappa_i^{c_i}\dots\\
&\dots  \boxtimes_1 \DR_{{g_{k}}}(a,-a)\lambda_{g_{k}} \psi_{e^{\prime\prime}_{k-1}}^{b_{k-1}}\psi_2^{d_2}\prod_i\kappa_i^{c_i}\\
&=  \Coef_{a^{2g_1}} \int_{\oM_{g_1,2}}\DR_{{g_1}}(a,-a)\lambda_{g_1}\psi_1^{d_1}\psi_{e^{\prime}_1}^{a_1}\prod_i\kappa_i^{c_i}
 \Coef_{a^{2g_2}} \int_{\oM_{g_2,2}} \DR_{{g_2}}(a,-a)\lambda_{g_2} \psi_{e^{\prime\prime}_1}^{b_1}\psi_{e^{\prime}_2}^{a_2}\prod_i\kappa_i^{c_i}\dots\\
 &\dots\Coef_{a^{2g_{k}}} \int_{\oM_{g_{k},2}}
 \DR_{{g_{k}}}(a,-a)\lambda_{g_{k}} \psi_{e^{\prime\prime}_{k-1}}^{b_{k-1}}\psi_2^{d_2}\prod_i\kappa_i^{c_i},
\end{align*}
 where we used the splitting property (\ref{eq:split}) of $\DR$-cycle.

Each tautological class is by definition represented by elements in a strata algebra which is a graded by codimention $\mathbb{Q}$-algebra with a basis given by isomorphism classes of pairs $[\Gamma,\gamma]$ where $\Gamma$ is a dual graph of the corresponding stratum of the moduli space and $\gamma$ is a product of $\psi$- and
$\kappa$-classes sitting on the vertices and half-edges of $\Gamma$, see (\cite{P16}) for survey.

Observe that an element $\psi_1^a\psi_2^b\prod_i\kappa_i^{c_i}$ can be written in the form 
$$
\psi_1^a\psi_2^b\prod_i\kappa_i^{c_i} = \omega_0+\bar{\omega}
$$
where $\omega_0$ is represented by a class $[1,\gamma]$ ($1$ is graph without edges) and $\bar{\omega}$ is represented by a boundary classes $[\Gamma,\gamma^{\prime}]$ where $\Gamma$ necessarily have edges.

By \cite[Theorem 1.1]{BSZ16}
\begin{align*}
R^{g-1}(\mathcal{M}_{g,2}) = \mathbb{Q}\langle \psi_1^{g-1},\psi_2^{g-1}\rangle. 
\end{align*}
Using this result it is immediate that 
\begin{align}
\label{arg}
\psi_1^a\psi_2^b\prod_i\kappa_i^{c_i} = \alpha \psi_1^{g-1}+\beta \psi_2^{g-1}+\bar{\omega},
\end{align}
where $\alpha,\beta \in  \mathbb{Q}$. Note that $\bar{\omega}$ being a boundary class means that its dual graph contains strictly bigger than one vertice.  As we discussed above, $\bar{\omega}$ contributes nontrivially only if it is of the form (\ref{bumboo}).
Hence we indeed have a recursion on genus.

A conjectural formula for $\lambda_g\DR_g(a,-a)$ is true in Gorenstein quotient

From the first two properties  (\ref{eq:Thm-irrdiv}), (\ref{eq:Thm-20div})  it is clear that again this integral 
\begin{align*}
\int_{\oM_{g,2}}\sfB^g\omega, \quad \omega \in R^{g-1}(\oM_{g,2}).
\end{align*}
is zero unless $\omega$ is of the form (\ref{bumboo}).
Now  using properties (\ref{eq:Thm-irrdiv}), (\ref{eq:Thm-20div}), (\ref{eq:Thm-Inter-1-g1-g2-2}) and repeating the argument (\ref{arg}) we get the same recursion on genus.

Recursion stops at genus 1 stage with equal outputs. Namely,
\begin{align*}
\int_{\oM_{1,2}} \sfB^1 = \int_{\oM_{1,2}}\psi_2^2 =  \dfrac{1}{24}
\end{align*}
and using Hain's formula for the restriction of $\DR$-cycle on compact type space $\mathcal{M}_{1,2}^{ct}$
\begin{align*}
 &\Coef_{a^{2}} \int_{\oM_{1,2}} \DR_1(a,-a)\lambda_1 =  \Coef_{a^{2}} \int_{\oM_{1,2}} \dfrac{a^2}{2}(\psi_1^{\prime}+\psi_2^{\prime}+2\delta_0^{12})\lambda_1\\
&=\int_{\delta_0^{12}}\lambda_1 = \dfrac{1}{24},
\end{align*}
where $\psi_i^{\prime}$ mean pullback of $\psi_i$ to  $\mathcal{M}_{1,2}^{ct}$ and $\delta_0^{12}$ is a class of divisor whose generic point is represented by a curve with nonseparating node.

 From the same recursion for $\DR_{g}(a,-a)\lambda_g$ and $\sfB^g$ sides the theorem \ref{th} follows.

\end{proof}

\section{Acknowledgements}
I am thankful to A.Buryak both for suggesting this problem and for valuable discussions. The author is partially supported by International Laboratory of Cluster Geometry NRU HSE, RF Government grant, ag. № 075-15-2021-608 dated 08.06.2021

\end{document}